\documentclass[12pt]{article}

\usepackage{bm}
\usepackage{tikz}
\usepackage{pgfplots} 
\usepackage{amsfonts}
\usepackage{amsmath, amssymb, amsthm}

\newtheorem{theorem}{Theorem}

\newtheorem{corollary}[theorem]{Corollary}

\newtheorem{lemma}[theorem]{Lemma}

\newtheorem{proposition}[theorem]{Proposition}

\begin{document}
%%%%% title : short title may not be used but TITLE is required.
% \title{TITLE}
% \title[short title]{TITLE}
\title{Analytic regularity for a singularly perturbed fourth order
reaction-diffusion boundary value problem}

\author{P. Constantinou  and C. Xenophontos
\footnote{Corresponding author. Email: xenophontos@ucy.ac.cy} \\
Department of Mathematics and Statistics\\
University of Cyprus\\
 PO BOX 20537\\Nicosia 1678\\
 Cyprus.
}

%%%%% AMS/Keywords %%%%%%%%%%%
%\ams%{65N30}    %AMS subject classification can be found in http://mathscinet.ams.org/msc/msc2010.html
%\keywords{singularly perturbed problem, reaction-convection-diffusion, boundary layers, analytic regularity.}    %List 3 to 8 keywords

%%%% maketitle %%%%%
\maketitle

%%%%% Begin Abstract %%%%%%%%%%%
\begin{abstract}
We consider a fourth order, reaction-diffusion type, singularly perturbed boundary
value problem, and the regularity of its solution. Specifically, we provide estimates for arbitrary 
order derivatves, which are explicit in the singular perturbation parameter as well as the 
differentiation order. Such estimates are needed for the numerical analysis of high order methods, e.g.
$hp$ Finite Element Method (FEM).
\end{abstract}
%%%%%  end %%%%%%%%%%%

%%%%%% Start %%%%%%%%%%%%%%%%%
\section{Introduction}

\label{intro}

Singularly perturbed problems (SPPs), and the numerical approximation of their
solution, have been studied extensively over the last few decades (see,
e.g., the books \cite{mos}, \cite{morton}, \cite{rst} and the references
therein). As is well known, a main difficulty in these problems is the
presence of \emph{boundary layers} in the solution, which appear due to fact
that the limiting problem (i.e. when the singular perturbation parameter
tends to 0), is of different order than the original one, and the (`extra')
boundary conditions can only be satisfied if the solution varies rapidly in
the vicinity of the boundary -- hence the name {\emph{boundary layers}}.

In most numerical methods, derivatives of the (exact) solution appear in
the error estimates, hence one should have a clear picture of how these
derivatives grow with respect to the singular perturbation parameter. For
low order numerical methods, such as Finite Differences (FD) or the $h$
version of the Finite Element Method (FEM), derivatives up to order 3 are
usually sufficient. For high order methods such as the $hp$ version of the
FEM, derivatives of arbitrary order are needed, thus knowing how these
behave with respect to the singular perturbation parameter as well as
the differentiation order, is necessary. Most commonly, second order 
SPPs are studied, but recently fourth order SPPs have gained in popularity (see, e.g., 
\cite{FR}, \cite{FRW}, \cite{PZMX}, \cite{X} and the references therein).
For the problem under consideration (see eq. (\ref{bvp}) ahead),
a boundary layer will be present in the \emph{derivative} of the solution, as 
opposed to the solution itself as is the case with second order SPPs.

In this article we prove two types of regularity for the solution: first we show
classical differentiability and then, using the method of matched asymptotic 
expansions (see, e.g. \cite{ms}), we refine the estimates by decomposing the solution into a 
smooth part, two boundary layers (one at each endpoint of the domain) and a
(negligible) remainder. Under the assuption of analytic input data, we give
detailed estimates on the behavior of the derivatives of each component in 
the decomposition. The material presented in this article is a modification of Ch. 2 in \cite{PC}.

The rest of the paper is organized as follows: in Section \ref{model} we
present the model problem and the regularity of its solution in terms of classical
differentiability. Section \ref{sec:DE} contains the asymptotic expansion for
the solution, under the assumption that the singular perturbation parameter
is small enough, and in Section \ref{sec:main} we present our main result,
which describes the analytic regularity of the solution.

With $I\subset \mathbb{R}$ an open, bounded interval with boundary $\partial I$ and
measure $\left\vert I\right\vert $, we will denote by $C^{k}(I)$ the space
of continuous functions on $I$ with continuous derivatives up to order $k$.
We will use the usual Sobolev spaces $H^{k}(I)=W^{k,2}$ of functions on $I $
with $0,1,2,...,k$ generalized derivatives in $L^{2}\left( I\right) $,
equipped with the norm and seminorm $\left\Vert \cdot \right\Vert _{k,I}$
and $\left\vert \cdot \right\vert _{k,I}$, respectively. The
usual $L^{2}(I)$ inner product will be denoted by $\left\langle \cdot ,\cdot
\right\rangle _{I}$, with the subscript omitted when there is no confusion.
We will also use the space 
\begin{equation*}
H_{0}^{2}\left( I\right) =\left\{ u\in H^{1}\left( I\right) :\left. u\right\vert _{\partial I }=
\left. u'\right\vert _{\partial I }=0\right\} .
\end{equation*}%
The norm of the space $L^{\infty }(I)$ of essentially bounded functions is
denoted by $\Vert \cdot \Vert _{\infty ,I}$. Finally, the notation
\textquotedblleft $a\lesssim b$\textquotedblright\ means \textquotedblleft $%
a\leq Cb$\textquotedblright\ with $C$ being a generic positive constant,
independent of any singular perturbation parameters.

%%%%%%%%%%%%%%%%%%%%%%%%%%%%%%%%%%%%%%

\section{The model problem and its regularity}\label{model}
We consider the following model problem: find $u \in C^4 (I), I=(0,1)$, such that
\begin{equation}
\label{bvp}
\left.\begin{array}{c}
\varepsilon^2 u^{(4)}(x)-\left(\alpha(x) u^{\prime}(x)\right)^{\prime}+\beta(x) u(x)=f(x), \text { for } x \in I \\
u(0)=u^{\prime}(0)=u(1)=u^{\prime}(1)=0
\end{array}\right\}
\end{equation}
where $\alpha(x)  > 0, \beta(x) \geq 0 \;  \forall  \; x \in \bar{I}$ 
and $f$ are given (analytic) functions. Specifically, we assume that there exist positive constants independent

\begin{equation}
\label{analytic}
\left\|\alpha^{(n)}\right\|_{\infty,I} \leq C_\alpha \gamma_\alpha^n n !, \;\left\|\beta^{(n)}\right\|_{\infty,I} \leq C_\beta \gamma_\beta^n n !, \; \left\|f^{(n)}\right\|_{\infty,I} \leq C_f \gamma_f^n n !.
\end{equation}

The solution to (\ref{bvp}) will be analytic, if the data is analytic, and in particular there holds
\begin{equation}
\label{cd}
\Vert u^{(n)} \Vert_{\infty,I} \lesssim K^n \max \{n^n, \varepsilon^{1-n} \},
\end{equation}
for some constant $K$, indepedent of $\varepsilon$ and $n$. This is shown in the same inductive way as
Theorem 1 in \cite{melenk97}.

The above estimate may be sharpened if one uses the method of asymptotic expansions (see, e.g., \cite{ms}). 
Anticipating that boundary layers will be present near the endpoints of $I$, we define
the stretched variables $\tilde{x} = x / \varepsilon$ , $\hat{x} = (1 - x) / \varepsilon$, and make
the formal ansatz
\begin{equation}
\label{ansatz}
u \sim \sum_{j=0}^{\infty} \varepsilon^j\left[U_j(x)+\tilde{U}_j(\tilde{x})+\hat{U}_j(\hat{x})\right],
\end{equation}
where the functions $U_j, \tilde{U}_j, \hat{U}_j$ will be determined shortly. We insert (\ref{ansatz}) in the
differential equation in (\ref{bvp}) and separate the slow (i.e. $x$) and fast (i.e. $\tilde{x}, \hat{x}$) variables.
Equating like powers of $\varepsilon$ on both sides of the resulting equation, we get
\begin{equation}
\label{eq:Uj}
\left.\begin{array}{ll}
-\left(\alpha(x) U_0^{\prime}(x)\right)^{\prime}+\beta(x) U_0(x)=f(x), & \\
-\left(\alpha(x) U_1^{\prime}(x)\right)^{\prime}+\beta(x) U_1(x)=0, & \\
-\left(\alpha(x) U_j^{\prime}(x)\right)^{\prime}+\beta(x) U_j(x)=-U_{j-2}^{(4)}(x), \quad j=2,3, \cdots
\end{array}\right\}
\end{equation}
\begin{equation}
\label{eq:Utj}
\left.\begin{array}{l}
\tilde{U}_0^{(4)}(\tilde{x})-\left(\tilde{\alpha}_0(\tilde{x}) \tilde{U}_0^{\prime}(\tilde{x})\right)^{\prime}=0, \\
\tilde{U}_1^{(4)}(\tilde{x})-\left(\tilde{\alpha}_0(\tilde{x}) \tilde{U}_1^{\prime}(\tilde{x})\right)^{\prime}=0, \\
\tilde{U}_j^{(4)}(\tilde{x})-\left(\tilde{\alpha}_0(\tilde{x}) \tilde{U}_j^{\prime}(\tilde{x})\right)^{\prime}=\tilde{F}_j(\tilde{x}), \quad j=2,3, \cdots
\end{array}\right\}
\end{equation}
and similarly for $\hat{U_j}$. In (\ref{eq:Utj}) we used the notation 
$$\tilde{\alpha}_k(\tilde{x}) = 
\frac{\alpha^{(k)}(0)}{k!}\tilde{x}^k \; , \;\tilde{\beta}_k(\tilde{x}) = 
\frac{\beta^{(k)}(0)}{k!}\tilde{x}^k
$$
and the right hand side functions are given by
$$
\tilde{F}_j(\tilde{x}):=\sum_{k=1}^{A_j} \tilde{\alpha}_k(\tilde{x}) \tilde{U}_{j-k}^{\prime \prime}(\tilde{x})+\sum_{k=1}^{A_j} \tilde{\alpha}_k^{\prime}(\tilde{x}) \tilde{U}_{j-k}^{\prime}(\tilde{x})-\sum_{k=0}^{B_j} \tilde{\beta}_k(\tilde{x}) \tilde{U}_{j-2-k}(\tilde{x}),
$$
where
$$
A_j=\left\{\begin{array}{ll}
\frac{j}{2}, & \text { if } j \text { is even, } \\
\frac{j-1}{2}, & \text { if } j \text { is odd }
\end{array} \quad B_j= \begin{cases}\frac{j-2}{2}, & \text { if } j \text { is even } \\
\frac{j-3}{2}, & \text { if } j \text { is odd }\end{cases}\right. .
$$
An analogous system as (\ref{eq:Utj}) is obtained for the functions $\hat{U}_j$
with $\hat{\alpha}_k(\hat{x}) = \frac{\alpha^{(k)}(1)}{k!}\hat{x}^k \; , \;\hat{\beta}_k(\hat{x}) = 
\frac{\beta^{(k)}(1)}{k!}\hat{x}^k$.
In order to satisfy the boundary conditions in (\ref{bvp}), the above systems of equations are supplemented with the following:
\begin{equation}
\label{bc_ij}
\left.
\begin{aligned}
& U_0(0)=U_0(1)=0,\\
& \tilde{U}_0(0)=0, \tilde{U}_j(0)=-U_j(0), j \geq 1, \\
& \tilde{U}_0^{\prime}(0)=0, \tilde{U}_{j+1}^{\prime}(0)=-U_j^{\prime}(0), j \geq 0, \\
& \lim _{\tilde{x} \rightarrow \infty} \tilde{U}_j(\tilde{x})=0, j \geq 0, \\
& \hat{U}_0(0)=0, \hat{U}_j(0)=-U_j(1), j \geq 1, \\
& \lim _{\hat{x} \rightarrow \infty} \hat{U}_j(\hat{x})=0, j \geq 0, \\
& \hat{U}_0^{\prime}(0)=0, \hat{U}_{j+1}^{\prime}(0)=-U_j^{\prime}(1), j \geq 0 .
\end{aligned}
\right\}
\end{equation}
The following table displays all the boundary value problems satisfied by 
$U_j, \tilde{U}_j, \hat{U}_j , j \in \mathbb{N}_0$.

\begin{table}[h]
\begin{center}
\begin{equation*}
\begin{array}{||l||l||l||}
\hline \text { Outer solution } & \text { Boundary layer at } x=0 & \text { Boundary layer at } x=1 \\
\hline \hline-\left(\alpha U_0^{\prime}\right)^{\prime}+\beta U_0=f, & \tilde{U}_0^{(4)}-\left(\tilde{\alpha}_0 \tilde{U}_0^{\prime}\right)^{\prime}=0, & \hat{U}_0^{(4)}-\left(\hat{\alpha}_0 \hat{U}_0^{\prime}\right)^{\prime}=0, \\
U_0(0)=-\tilde{U}_0(0)=0, & \lim _{\tilde{x} \rightarrow \infty} \tilde{U}_0(\tilde{x})=0, & \lim _{\hat{x} \rightarrow \infty} \hat{U}_0(\hat{x})=0, \\
U_0(1)=-\hat{U}_0(0)=0, & \tilde{U}_0^{\prime}(0)=0, & \hat{U}_0^{\prime}(0)=0, \\
\hline \hline-\left(\alpha U_1^{\prime}\right)^{\prime}+\beta U_1=0, & \tilde{U}_1^{(4)}-\left(\tilde{\alpha}_0 \tilde{U}_1^{\prime}\right)^{\prime}=0, & \hat{U}_1^{(4)}-\left(\hat{\alpha}_0 \hat{U}_1^{\prime}\right)^{\prime}=0, \\
U_1(0)=-\tilde{U}_1(0), & \lim _{\tilde{x} \rightarrow \infty} \tilde{U}_1(\tilde{x})=0, & \lim _{\hat{x} \rightarrow \infty} \hat{U}_1(\hat{x})=0, \\
U_1(1)=-\hat{U}_1(0), & \tilde{U}_1^{\prime}(0)=-U_0^{\prime}(0), & \hat{U}_1^{\prime}(0)=-U_0^{\prime}(1), \\
\hline \hline-\left(\alpha U_j^{\prime}\right)'+\beta U_j=-U_{j-2}^{(4)}, & \tilde{U}_j^{(4)}-\left(\tilde{\alpha}_0 \tilde{U}_j^{\prime}\right)'=\tilde{F}_j, & \hat{U}_j^{(4)}-\left(\hat{\alpha}_0 \hat{U}_j^{\prime}\right)^{\prime}=\hat{F}_j, \\
U_j(0)=-\tilde{U}_j(0), & \lim _{\tilde{x} \rightarrow \infty} \tilde{U}_j(\tilde{x})=0, & \lim _{\hat{x} \rightarrow \infty} \hat{U}_j(\hat{x})=0, \\
U_j(1)=-\hat{U}_j(0), & \tilde{U}_j^{\prime}(0)=-U_{j-1}^{\prime}(0), & \hat{U}_j^{\prime}(0)=-U_{j-1}^{\prime}(1) . \\
\hline
\end{array}
\end{equation*}
\end{center}
\caption{\label{table1}
Different boundary value problems satisfied by $U_j, \tilde{U}_j, \hat{U}_j , j \in \mathbb{N}_0$.}
\end{table}

\vspace{0.25cm}

Next, we define for $M \ \in \mathbb{N}$, the outer (smooth) expansion $u_M^S$ as
\begin{equation}
\label{uSM}
u_M^S (x) := \sum_{j=0}^M \varepsilon^j U_j(x),
\end{equation}
the boundary layer expansion at the left endpoint $\tilde{u}_M^{BL} (x)$ as
\begin{equation}
\label{utM}
\tilde{u}_M^{BL} (x) := \sum_{j=0}^M \varepsilon^j \tilde{U}_j(\tilde{x}),
\end{equation}
and the boundary layer expansion at the right endpoint $\hat{u}_M^{BL} (x)$ as
\begin{equation}
\label{uhM}
\hat{u}_M^{BL} (x) := \sum_{j=0}^M \varepsilon^j \hat{U}_j(\hat{x}).
\end{equation}
Finally, we define the remainder $r_M$ as
\begin{equation}
\label{rM}
r_M := u - \left( u_M^S + \tilde{u}_M^{BL} + \hat{u}_M^{BL} \right),
\end{equation}
and we have the \emph{decomposition}
\begin{equation}\label{decomp}
u = u_M^S + \tilde{u}_M^{BL} + \hat{u}_M^{BL} + r_M.
\end{equation}

\section{Derivative Estimates}\label{sec:DE}
In this section we present regularity results for each term in the decomposition (\ref{decomp}). Due to (\ref{bc_ij}),
we have to analyze the `triple' $\left( U_j, \tilde{U}_j, \hat{U}_j\right)$ simultaneously. For this we will need some
preliminary results, which we present next.
\subsection{Preliminaries}

\begin{proposition}
\label{prop:aux}
Let $w$ be the solution of the BVP
\begin{eqnarray*}
-(\alpha(x) w'(x))' + \beta(x)w(x) &=& g(x) \; , \; x \in I, \\
w(0) = g^{-} \in \mathbb{R} \; , \; w(1)&=&g^{+} \in \mathbb{R},
\end{eqnarray*}
where $g$ is an analytic function satisfying bounds similar to (\ref{analytic}). Then
there exist positive constants $C, \gamma$, such that
\[
\Vert u^{(n)} \Vert_{\infty, I} \leq C \left( |g^{-}| + |g^{+}| \right) \gamma^n n! \quad \forall \; n \in \mathbb{N}_0.
\]
\end{proposition}
\begin{proof}
This is a classical result, see, e.g. \cite{GLC}.
\end{proof}

\begin{proposition}\label{prop:MXO}
For $C_1, d>0$ and $\rho \ge 0$, the following estimates are valid with
$\tilde{\rho}=\rho / \varepsilon, \gamma=2 \max \left\{1, C_1^2\right\}$:
$$
\begin{aligned}
& \left(C_1 l+\tilde{\rho}\right)^{2 l} \leq 2^l\left(C_1 l\right)^{2 l}+2^l \tilde{\rho}^{2 l} \leq \gamma^l\left(l^{2 l}+\tilde{\rho}^{2 l}\right), \\
& \sup _{\rho>0} \rho^n \mathrm{e}^{-\frac{d \rho}{4}} \leq\left(\frac{4 n}{\mathrm{e}d}\right)^n .
\end{aligned}
$$
\end{proposition}
\begin{proof}
This was shown in \cite[Lemma 51]{MXO}.
\end{proof}

\begin{proposition}
\label{prop:lem_aux2}
Let $j \in \mathbb{N}, \kappa, \lambda \in \mathbb{C}^+$, with $\lambda = \kappa^2$ and let $F$ be an entire 
function satisfying, for some $C_F > 0$ and $q > \frac{4j}{|\kappa|}$,
$$
|F(z)| \leq C_F \mathrm{e}^{-\operatorname{Re}(\kappa z)}(q+|z|)^{2 j-1}, \quad \text { for all } z \in \mathbb{C}.
$$
Furthermore, let $g_1 \in \mathbb{C}$ and let $w:(0,\infty) \rightarrow \mathbb{C}$ be the solution of
$$
w^{(4)}-\lambda w^{(2)}=F \text { on }(0, \infty), \quad w^{\prime}(0)=g_1, \lim _{x \rightarrow \infty} w(x)=0.
$$
Then, $w$ can be extended to an entire function (denoted again by $w$) which satisfies
$$
|w(z)| \leq C\left[\frac{C_F}{2 j}(q+|z|)^{2 j}+\frac{\left|g_1\right|}{|\kappa|}\right] \mathrm{e}^{-R e(\kappa z)}, \quad \text { for all } z \in \mathbb{C}.
$$
\end{proposition}
\begin{proof}
The proof follows very closely the proof of Lemma 7.3.6 in \cite{melenk} and Lemma 4
in \cite{PZMX}. The details appear in the Appendix. 
\end{proof}

\begin{lemma}\label{lem:V}
Let $j \in \mathbb{N}$ and let $V$ be an entire function which satisfies
$$
|V(z)| \leq \frac{C \gamma_V^j}{(j-1) !}\left(q_j+|z|\right)^{2(j-1)} \mathrm{e}^{-\operatorname{Re}(\sqrt{\alpha_0} z)}, \forall z \in \mathbb{C},
$$
for some constants $C, \gamma_V, \alpha_0$, and with $q_j = \frac{4j}{\sqrt{\alpha_0}}$. Then for all 
$n \in \mathbb{N}$, we have
$$
\left|V^{(n)}(z)\right| \leq C \frac{n ! \mathrm{e}^{n+1}}{(n+1)^n} \frac{\gamma_V^j}{(j-1) !} \mathrm{e}^{-\operatorname{Re}(\sqrt{\alpha_0} z)}\left(q_j+\frac{n+1}{\sqrt{\alpha_0}}+|z|\right)^{2(j-1)}, \; z \in \mathbb{C}.
$$
\end{lemma}
\begin{proof}
We apply Cauchy’s Integral Theorem for derivatives, letting the integration
path be a circle of radius $\frac{n+1}{\sqrt{\alpha_0}}$. Then, for all $n \in \mathbb{N}$, we have
$$
\begin{aligned}
\left|V^{(n)}(z)\right| & \leq \frac{n !}{2 \pi} \oint_D \frac{|V(\zeta)|}{|\zeta-z|^{n+1}} d \zeta \\
& \leq C \frac{(\sqrt{\alpha_0})^n n ! e^{n+1}}{(n+1)^n} \frac{\gamma_V^j}{(j-1) !} \mathrm{e}^{-R e(\sqrt{\alpha_0} z) \mid}\left(q_j+\frac{n+1}{\sqrt{\alpha_0}}+|z|\right)^{2(j-1)}.
\end{aligned}
$$
\end{proof}

%%%%%%%%%%%%%%%%%%%%%%%%%%%%%%%%%%%%%%%%%%%%%%%
\subsection{Main result}\label{sec:main}
We first have the following.
\begin{theorem}\label{thm:main0}
Consider the triple $\left( U_j, \tilde{U}_j, \hat{U}_j \right)$, as defined in Table 1. 
Then, there exist positive constants $a, \gamma_*, C, K_f > 1$, independent of $\varepsilon$, such that for
all $n \in \mathbb{N}$,
\begin{equation}
\label{Ujn}
\left\|U_j^{(n)}\right\|_{\infty,I} \leq C \gamma_*^j \frac{a^{2 j} j^{2 j}}{j !} K^n n !,  \forall j \in \mathbb{N}_0, 
\end{equation}
\begin{eqnarray}
\left|\tilde{U}_j(z)\right| &\leq& C \gamma_*^j \frac{1}{(j-1) !}(a j+|z|)^{2(j-1)} \mathrm{e}^{-\sqrt{\alpha(0)} R e(z)},  \forall z \in \mathbb{C}, j \in \mathbb{N},  \label{Utj} \\
\left|\hat{U}_j(z)\right| &\leq& C \gamma_*^j \frac{1}{(j-1) !}(a j+|z|)^{2(j-1)} \mathrm{e}^{-\sqrt{\alpha(1)} R e(z)},  \forall z \in \mathbb{C}, j \in \mathbb{N} \label{Uhj}.
\end{eqnarray}
\end{theorem}
\begin{proof}
Let $\gamma$ be the constant of Proposition \ref{prop:aux} and let  $\gamma_f$ be given in (\ref{analytic}). We choose
$a, \gamma_*>1$, $\tilde{\gamma}_f > \max \{ \gamma, \gamma_f \}$ and $K=\max \{ 1, e \tilde{\gamma}_f\}$
 to satisfy
$$
\begin{gathered}
a \geq \frac{3}{\min \{\sqrt{\alpha(0)}, \sqrt{\alpha(1)}\}}, \\
\gamma_* \geq \max \left\{\gamma_\alpha, \gamma_{\alpha^{\prime}}, \gamma_\beta\right\}, \\
\frac{1}{\gamma_*}\left(1+\frac{K}{\min \{\sqrt{\alpha(0)}, \sqrt{\alpha(1)}\}}\right) \leq 1, \\
\frac{2}{a^2}\left(\frac{12 K^4}{a^2 \gamma_*^2}+1\right) \leq 1 .
\end{gathered}
$$
We then proceed with an induction argument. By Proposition \ref{prop:aux},
\[
\Vert U_0 \Vert_{\infty, I} \leq C K^n n! ,
\]
hence (\ref{Ujn}) holds for $j=0$. Next we find that $\tilde{U}_0 = \hat{U}_0 = 0$. For $\tilde{U}_1, \hat{U}_1$
we calculate
$$
\tilde{U}_1(\tilde{x})=\frac{U_0^{\prime}(0)}{\sqrt{\alpha(0)}} \mathrm{e}^{-\sqrt{\alpha(0) \tilde{x}}}
$$
and
$$
\hat{U}_1(\hat{x})=\frac{U_0^{\prime}(1)}{\sqrt{\alpha(1)}} \mathrm{e}^{-\sqrt{\alpha(1)} \hat{x}} .
$$
This shows that (\ref{Utj}), (\ref{Uhj}) hold for $j=1$, once we extend $\tilde{U}_1, \hat{U}_1$ to $\mathbb{C}$.
Moreover,
$$
\left|\tilde{U}_1(0)\right|+\left|\hat{U}_1(0)\right| \leq \frac{\left|U_0^{\prime}(0)\right|}{\sqrt{\alpha(0)}}+\frac{\left|U_0^{\prime}(1)\right|}{\sqrt{\alpha(1)}} \leq C K\left(\frac{1}{\sqrt{\alpha(0)}}+\frac{1}{\sqrt{\alpha(1)}}\right),
$$
and by Proposition \ref{prop:aux} 
\[
\Vert U_1 \Vert_{\infty, I} \leq C K^n n!.
\]
We next assume that (\ref{Ujn})--(\ref{Uhj}) hold for $j$ and show them for $j+1$. Since $\tilde{U}_{j+1}$
satisfies (\ref{eq:Utj}), we may use Proposition \ref{prop:lem_aux2}, once we show that 
$$
\left|\tilde{F}_{j+1}(z)\right| \leq C \gamma_*^j e^{-\sqrt{\alpha(0)} \operatorname{Re}(z)} \frac{1}{(j-1) !}\left(t_{j+1}+|z|\right)^{2 j-1},
$$
for some number $t_{j+1}$, depending on $j $. To this end, we recall that
$$
\tilde{F}_j(z)=\sum_{k=1}^{A_j} \alpha_k(z) \tilde{U}_{j-k}^{\prime \prime}(z)+\sum_{k=1}^{A_j} \alpha_k^{\prime}(z) \tilde{U}_{j-k}^{\prime}(z)-\sum_{k=0}^{B_j} \beta_k(z) \tilde{U}_{j-2-k}(z).
$$
By the induction hypothesis we have that, for all $k = 0, \ldots, j - 1$,
$$
\left|\tilde{U}_{j-k}(\tilde{z})\right| \leq C \gamma_*^{j-k} \frac{1}{(j-k-1) !}(a(j-k)+|z|)^{2(j-k-1)} \mathrm{e}^{-\sqrt{\alpha(0)} \operatorname{Re}(z)}.
$$
From Lemma \ref{lem:V} we obtain, using the analyticity of the functions $\alpha, \beta$,
\begin{eqnarray*}
\left|\tilde{F}_{j+1}(z)\right| &\leq& C \mathrm{e}^{-\sqrt{\alpha(0)} \operatorname{Re}(z)}
\sum_{k=1}^{A_{j+1}} \gamma_*^{j+1-k} \left( C_{\alpha} \gamma_\alpha^k + C_{\beta} \gamma_{\alpha'}^k\right)
 |z|^k \frac{1}{(j-k) !}\times \\
&\times& \left( a(j+1-k)+\frac{3}{\sqrt{\alpha(0)}}+|z|\right)^{2(j-k)} + 
C \sum_{k=0}^{B_{j+1}}  \gamma_*^{j-1-k} C_\beta \gamma_\beta^k|z|^k \frac{1}{(j-2-k) !} \times \\
&\times& \left(a(j-1-k)+\frac{3}{\sqrt{\alpha(0)}}+|z|\right)^{2(j-2-k)} \mathrm{e}^{-\sqrt{\alpha(0)} \operatorname{Re}(z)}.
\end{eqnarray*}
The following inequalities will be used in the sequel, and follow by elementary considerations: for $j \ge 2$, we have
$$
\begin{aligned}
\frac{1}{(j-k) !} & \leq \frac{(j-1)^{k-1}}{(j-1) !} \leq \frac{1}{(\mid j-1) !}\left(a j+\frac{3}{\sqrt{\alpha(0)}}+|z|\right)^{k-1}, k \leq j-1, \\
\frac{1}{(j-2-k) !} & \leq \frac{(j-1)^{k+1}}{(j-1) !} \leq \frac{1}{(j-1) !}\left(a j+\frac{3}{\sqrt{\alpha(0)}}+|z|\right)^{k+1}, k \leq j-2,
\end{aligned}
$$
since $a > 1$ and hence, for $k \leq j-1$, there holds
$$
|z|^k \frac{1}{(j-k) !}\left(a(j+1-k)+\frac{3}{\sqrt{\alpha(0)}}+|z|\right)^{2(j-k)} \leq
$$
$$
\leq \frac{1}{(j-1) !}\left(a(j+1-k)+\frac{3}{\sqrt{\alpha(0)}}+|z|\right)^{2 j-1}.
$$
This gives
$$
\begin{aligned}
& \left|\tilde{F}_{j+1}(z)\right| \leq C \gamma_*^j \frac{1}{(j-1) !}\left(a j+\frac{3}{\sqrt{\alpha(0)}}+|z|\right)^{2 j-1} \mathrm{e}^{-\sqrt{\alpha(0)} \operatorname{Re}(z)} \times \\
& \times\left[\sum_{k=1}^{A_{j+1}} C_1 C_\alpha \gamma_\alpha\left(\frac{\gamma_\alpha}{\gamma_*}\right)^{k-1}+\sum_{k=1}^{A_{j+1}} C_2 C_{\alpha^{\prime}} \gamma_{\alpha^{\prime}}\left(\frac{\gamma_{\alpha^{\prime}}}{\gamma_*}\right)^{k-1}++\sum_{k=0}^{B_{j+1}} \frac{C_3 C_\beta}{\gamma_\beta}\left(\frac{\gamma_\beta}{\gamma_*}\right)^k\right] . \\
&
\end{aligned}
$$
The choice of the constants ensures that the sums above are bounded by convergent geometric series, thus
$$
\left|\tilde{F}_{j+1}(z)\right| \leq C \gamma_*^j \frac{1}{(j-1) !}\left(a j+\frac{3}{\sqrt{\alpha(0)}}+|z|\right)^{2 j-1} \mathrm{e}^{-\sqrt{\alpha(0)} R e(z)}.
$$
By Proposition \ref{prop:lem_aux2}, we get
$$
\left|\tilde{U}_{j+1}(z)\right| \leq C\left[\frac{\gamma_*^j}{j} \frac{1}{(j-1) !}\left(a j+\frac{3}{\sqrt{\alpha(0)}}+|z|\right)^{2 j}+\left|\frac{U_j^{\prime}(0)}{\sqrt{\alpha(0)}}\right|\right] \mathrm{e}^{-\sqrt{\alpha(0)} \operatorname{Re}(z)}
$$
We bound the term $\left\vert U'_j(0) \right\vert$ using the induction hypothesis corresponding to (\ref{Ujn}),
and we obtain
\begin{eqnarray*}
\left|\tilde{U}_{j+1}(z)\right| &\leq& C\left[\frac{\gamma_*^j}{j !}\left(a j+\frac{3}{\sqrt{\alpha(0)}}+|z|\right)^{2 j}+\frac{K \gamma_*^j a^{2 j} j^{2 j}}{j ! \sqrt{\alpha(0)}}\right] \mathrm{e}^{-\sqrt{\alpha(0)} \operatorname{Re}(z)}\\
&\leq& C \gamma_*^{j+1} \frac{1}{j !}\left(a j+\frac{3}{\sqrt{\alpha(0)}}+|z|\right)^{2 j}\left[\frac{1}{\gamma}+\frac{K}{\sqrt{\alpha(0)} \gamma_*}\right] \mathrm{e}^{-\sqrt{\alpha(0)} \operatorname{Re}(z)} \\
&\leq& C \frac{\gamma^{j+1}}{j !}(a(j+1)+|z|)^{2 j} \mathrm{e}^{-\sqrt{\alpha(0)} \operatorname{Re}(z)}.
\end{eqnarray*}
This establishes (\ref{Utj}) for $j+1$. In the same manner we may show that (\ref{Uhj}) holds for $j+1$.

We next turn our attention to $U_{j+1}$, which satisfies (\ref{eq:Uj}) and (\ref{bc_ij}).
The induction hypothesis for $j-1$, gives
$$
\left\|U_{j-1}^{(4)}\right\|_{\infty,I} \leq C \gamma_*^{j-1} \frac{a^{2(j-1)}(j-1)^{2(j-1)}}{(j-1) !} K^4 4 !
$$
Proposition \ref{prop:aux} then yields the estimate
\[
\Vert U_{j+1}^{(n)} \Vert_{\infty, I} \leq C \left( | \tilde{U}_{j+1}(0)| + |\tilde{U}_{j+1}(0)| \right) \gamma_*^n n! \quad \forall \; n \in \mathbb{N}_0,
\]
and we have already shown that
$$
\left|\tilde{U}_{j+1}(0)\right| + \left|\hat{U}_{j+1}(0)\right|  \leq C \gamma_*^{j+1} \frac{1}{j !} a^{2 j}(j+1)^{2 j}.
$$
Thus
$$
\begin{aligned}
& \left\|U_{j-1}^{(4)}\right\|_{L^{\infty}(I)}+\left|\tilde{U}_{j+1}(0)\right|+\left|\hat{U}_{j+1}(0)\right| \leq \\
& \quad \leq C\left[4 ! K^4 \gamma_*^{j-1} \frac{1}{(j-1) !} a^{2(j-1)}(j-1)^{2(j-1)}+2 \gamma_*^{j+1} \frac{1}{j !} a^{2 j}(j+1)^{2 j}\right] \\
& \quad \leq C \gamma_*^{j+1} \frac{a^{2(j+1)}(j+1)^{2(j+1)}}{(j+1) !}\left[\frac{4 ! K^4}{a^4 \gamma_*^2}+\frac{2}{a^2}\right] \\
& \quad \leq C \gamma_*^{j+1} \frac{a^{2(j+1)}(j+1)^{2(j+1)}}{(j+1) !} .
\end{aligned}
$$
This gives (\ref{Ujn}) for $j+1$, and the proof is complete.
\end{proof}

The above result allows us to establish the following.

\begin{corollary}\label{cor_BLn}
Let the functions $\tilde{U}_j, \hat{U}_j$ be defined by (\ref{eq:Utj})--(\ref{bc_ij}). Then, there exist constants
$C, \tilde{K}, \hat{K}, \gamma_* > 0$, independent of $j$ and $n$, such that for all $\tilde{x}, \hat{x} > 0$
and $\forall \; n\in \mathbb{N}$ there holds
$$
\left|\tilde{U}_j^{(n)}(\tilde{x})\right| \leq C \tilde{K}^n\left(a^2 \mathrm{e} \gamma\right)^j j^{j-1} \mathrm{e}^{-\frac{\sqrt{\alpha(0)}}{2}|\tilde{x}|}
$$
$$
\left|\hat{U}_j^{(n)}(\hat{x})\right| \leq C \hat{K}^n\left(a^2 \mathrm{e} \gamma\right)^j j^{j-1} \mathrm{e}^{-\frac{\sqrt{\alpha(1)}}{2}|\hat{x}|}
$$
\end{corollary}
\begin{proof}
We will only show the result for $\tilde{U}_j$ since $\hat{U}_j$ is similar. We have already shown that $\tilde{U}_j$ is
entire and by Proposition \ref{lem:V} we get, for $z \in \mathbb{C}$,
$$
\left|\tilde{U}_j^{(n)}(z)\right| \leq C \frac{(\sqrt{\alpha(0)})^n n^n \mathrm{e}^{n+1}}{(n+1)^n} \gamma_*^j \mathrm{e}^{-\operatorname{Re}(\sqrt{\alpha(0)} z)} \frac{1}{(j-1) !}\left(a j+\frac{n+1}{\sqrt{\alpha(0)}}+|z|\right)^{2(j-1)} .
$$
By Proposition \ref{prop:MXO}, we have for $\tilde{x} > 0$,
\begin{eqnarray*}
\left(a j+\frac{n+1}{\sqrt{\alpha(0)}}+\tilde{x}\right)^{2(j-1)} \mathrm{e}^{-\sqrt{\alpha(0)}|\tilde{x}|} &\leq& 
C\left(a j+\frac{n+1}{\sqrt{\alpha(0)}}\right)^{2(j-1)} \mathrm{e}^{-\frac{\sqrt{\alpha(0)}}{2}|\tilde{x}|}
\end{eqnarray*}
$$
\begin{aligned}
& \leq C(a j)^{2(j-1)}\left(1+\frac{(n+1)}{\sqrt{\alpha(0)} a j}\right)^{2(j-1)} \mathrm{e}^{-\frac{\sqrt{\alpha(0)}}{2}}|\tilde{x}| \\
& \leq C(a j)^{2(j-1)}\left(1+\frac{2(n+1)}{2 a \sqrt{\alpha(0)}(j-1)}\right)^{2(j-1)} \mathrm{e}^{-\frac{\sqrt{\alpha(0)}}{2}|\tilde{x}|} \\
& \leq C(a j)^{2(j-1)} \mathrm{e}^{\frac{2(n+1)}{a \sqrt{\alpha(0)}}} \mathrm{e}^{-\frac{\sqrt{\alpha(0)}}{2}}|\tilde{x}|, 
\end{aligned}
$$
which yields, with the help of Stirling's formula,
$$
\begin{aligned}
& \frac{1}{(j-1) !}\left(a j+\frac{n+1}{\sqrt{\alpha(0)}}+\tilde{x}\right)^{2(j-1)} \mathrm{e}^{-\sqrt{\alpha(0)}|\tilde{x}|} \leq \\
& \leq C a^{2(j-1)} \frac{j^{2(j-1)}}{(j-1) !} \mathrm{e}^{\frac{2(n+1)}{a \sqrt{\alpha(0)}}} \mathrm{e}^{-\frac{\sqrt{\alpha(0)}}{2}}|\tilde{x}| \\
& \leq C\left(a^2 \mathrm{e}\right)^{j-1} j^{j-1} \mathrm{e}^{\frac{2(n+1)}{a \sqrt{\alpha(0)}}} \mathrm{e}^{-\frac{\sqrt{\alpha(0)}}{2}}|\tilde{x}| .
\end{aligned}
$$
Therefore,
$$
\begin{aligned}
\left\vert \tilde{U}^{(n)}_j (z) \right\vert \leq C \frac{\mathrm{e} \sqrt{\alpha(0)}}{a^2} \mathrm{e}^{\frac{2(n+1)}{a \sqrt{\alpha(0)}}}\left(\frac{n \mathrm{e}}{n+1}\right)^n\left(a^2 \mathrm{e} \gamma\right)^j j^{j-1} \mathrm{e}^{-\frac{\sqrt{\alpha(0)}}{2}}|\tilde{x}|,
\end{aligned}
$$
and this completes the proof.
\end{proof}

We are finally in the position to prove our main result. The theorem below
tells us that the solution has an analytic character, the boundary layers do not
affect the solution in areas away from the boundary, and the remainder is exponentially
small.

\begin{theorem}\label{thm:main}
Assume (\ref{analytic}) holds and $u$ be the solution of (\ref{bvp}). Then there exist positive constants
$C, K, K_1, q$, independent of $\varepsilon$, such  that 
\[
u = u_M^S + \tilde{u}_M^{BL} + \hat{u}_M^{BL} + r_M,
\]
with the following being true $\forall \; n \in \mathbb{N}$ and $\forall \; x \in \overline{I}$:
$$
\left\|\left(u_M^S\right)^{(n)}\right\|_{0, I} \leq C K^n n !, 
$$
$$
\left|\left(\tilde{u}_M^{B L}\right)^{(n)}(x)\right|+\left|\left(\hat{u}_M^{B L}\right)^{(n)}(x)\right| \leq C K_1^n \varepsilon^{1-n} \mathrm{e}^{-\min \left\{\frac{\sqrt{\alpha(0)}}{2}, \frac{\sqrt{\alpha(1)}}{2}\right\} \operatorname{dist}(x, \partial I) / \varepsilon},
$$
$$
\left\|r_M\right\|_{L^{\infty}(\partial I)}+\left\|r_M^{\prime}\right\|_{L^{\infty}(\partial I)}+\left\|r_M\right\|_{0, I}
+ \varepsilon \Vert r'_M \Vert_{0,I} \leq C \mathrm{e}^{-q / \varepsilon},
$$
where $M$ is chosen so that $a^2 \varepsilon \mathrm{e} \gamma_* (M+1) < 1$, where $a$ and $\gamma_*$ are 
given in Theorem \ref{thm:main0}.
\end{theorem}
\begin{proof}
We have from Theorem \ref{thm:main0},
$$
\begin{aligned}
\left\|\left(u_M^S\right)^{(n)}\right\|_{0, I} & \leq \sum_{j=0}^M \varepsilon^j\left\|U_j^{(n)}\right\|_{0, I} \leq C \sum_{j=0}^M \varepsilon^j \gamma_*^j \frac{a^{2 j} j^{2 j}}{j !} K^n n ! \\
& \leq C K^n n ! \sum_{j=0}^M(\varepsilon \gamma_*)^j \frac{a^{2 j} j^{2 j}}{j !} \leq C K^n n ! \sum_{j=0}^M\left(a^2 \varepsilon \mathrm{e} \gamma_* j\right)^j \\
& \leq C K^n n ! \sum_{j=0}^{\infty}\left(a^2 \varepsilon \mathrm{e} \gamma_* M\right)^j \\
& \leq C K^n n ! 
\end{aligned}
$$
where we used the fact that the sum is a convergent geometric series, due to the assumption $a^2 \varepsilon \mathrm{e} \gamma_* (M+1) < 1$. Similarly, from Corollary \ref{cor_BLn},
$$
\begin{aligned}
& \left|\left(\tilde{u}_M^{B L}\right)^{(n)}(\tilde{x})\right| \leq \sum_{j=0}^{M+1} \varepsilon^j\left|\tilde{U}_j^{(n)}(\tilde{x})\right| \leq \sum_{j=1}^{M+1} \varepsilon^j \mathrm{e}^{-\frac{\sqrt{\alpha(0)}}{2}}|\tilde{x}| \tilde{K}^n\left(a^2 \mathrm{e} \gamma_*\right)^j j^{j-1} \\
& \leq \varepsilon \gamma_* \tilde{K}^n \mathrm{e}^{-\frac{\sqrt{\alpha(0)}}{2}}|\tilde{x}| \sum_{j=1}^{M+1}\left(a^2 \varepsilon \mathrm{e} \gamma_* (M+1)\right)^{j-1} \\
& \leq C \varepsilon \tilde{K}^n \mathrm{e}^{-\frac{\sqrt{\alpha(0)}}{2}} \tilde{x} \mid \sum_{j=0}^{\infty}\left(a^2 \varepsilon \mathrm{e} \gamma_* (M+1)\right)^j \\
& \leq C \varepsilon \tilde{K}^n \mathrm{e}^{-\frac{\sqrt{\alpha(0)}}{2}|\tilde{x}|} .
\end{aligned}
$$
In the same manner one can infer a similar result for $\hat{u}_M^{BL}$. 

Finally, we note that
$$
\begin{aligned}
r_M(0) & =u(0)-\left(\sum_{j=0}^M \varepsilon^j U_j(0)+\sum_{j=0}^{M+1} \varepsilon^j \tilde{U}_j(0)+\sum_{j=0}^{M+1} \varepsilon^j \hat{U}_j(1 / \varepsilon)\right) \\
& =-\varepsilon^{M+1} \tilde{U}_M(0)-\sum_{j=0}^{M+1} \varepsilon^j \hat{U}_j(1 / \varepsilon),
\end{aligned}
$$
since $u(0)=0, \tilde{U}_j(0) + \hat{U}_j(0) = 0$ for $j\ge 1$. Hence, using Theorem \ref{thm:main0} we have
$$
\begin{aligned}
\left|r_M(0)\right| & \leq \varepsilon^{M+1}\left|\tilde{U}_{M+1}(0)\right|+\sum_{j=1}^{M+1} \varepsilon^j\left|\hat{U}_j(1 / \varepsilon)\right| \\
& \leq C\left(a^2 \varepsilon \mathrm{e} \gamma_*\right)^{M+1}(M+1)^M+\sum_{j=1}^{M+1} \varepsilon^j C \mathrm{e}^{-\frac{\sqrt{\alpha(1)}}{2 \varepsilon}} \gamma_*^j j^{j-1} \\
& \leq C\left(a^2 \varepsilon \mathrm{e} \gamma_*\right)^{M+1}(M+1)^M+C \varepsilon \gamma_* \mathrm{e}^{-\frac{\sqrt{\alpha(1)}}{2 \varepsilon}} \sum_{j=1}^{M+1}\left(a^2 \varepsilon \mathrm{e} \gamma_*(M+1)\right)^{j-1} \\
& \leq C\left(a^2 \varepsilon \mathrm{e} \gamma_*\right)^{M+1}(M+1)^M+C \varepsilon \gamma_* \mathrm{e}^{-\frac{\sqrt{\alpha(1)}}{2 \varepsilon}} \sum_{j=0}^{\infty}\left(a^2 \varepsilon \mathrm{e} \gamma_*(M+1)\right)^j \\
& \leq C \varepsilon \gamma_*\left(a^2 \varepsilon \gamma_*(M+1)\right)^M+C \varepsilon \mathrm{e}^{-q \frac{\sqrt{\alpha(1)}}{2 \varepsilon}}
\end{aligned}
$$
for some positive contant $q$, independent of $\varepsilon$. Furthermore, 
$$
\begin{aligned}
r_M^{\prime}(0) & =u^{\prime}(0)-\left(\sum_{j=0}^M \varepsilon^j U_j^{\prime}(0)+\sum_{j=1}^{M+1} \varepsilon^{j-1} \tilde{U}_j^{\prime}(0)-\sum_{j=1}^{M+1} \varepsilon^{j-1} \hat{U}_j^{\prime}(1 / \varepsilon)\right) \\
& =-\sum_{j=1}^{M+1} \varepsilon^{j-1} \hat{U}_j^{\prime}(1 / \varepsilon),
\end{aligned}
$$
since $u^{\prime}(0)=\tilde{U}_0^{\prime}(0)=\hat{U}_0^{\prime}(0)=0, \tilde{U}_j^{\prime}(0)+U_{j-1}^{\prime}(0)=0$. Thus,
$$
\left|r_M^{\prime}(0)\right| \leq \sum_{j=1}^{M+1} \varepsilon^{j-1}\left|\hat{U}_j^{\prime}(1 / \varepsilon)\right| 
\leq C \hat{K} \mathrm{e}^{-q \sqrt{\alpha(1)} / \varepsilon},
$$
where Corollary \ref{cor_BLn}  was used. In the same way we obtain analogous results for $|r_M(1)|, |r'_M(1)|$.

We now apply the operator $
L:=\varepsilon^2 \frac{d^4}{d x^4}-\frac{d}{d x}\left(\alpha(x) \frac{d^2}{d x^2}\right)+\beta(x)
$ to the function $u-u^S_M$ and obtain,
$$
\begin{aligned}
L\left(u-u_M^s\right) & =f-\sum_{j=0}^M \varepsilon^j L\left(U_j\right) \\
& =f-\sum_{j=0}^M \varepsilon^j\left(\varepsilon^2 U_j^{(4)}-\left(\alpha(x) U_j^{\prime}\right)^{\prime}+\beta(x) U_j\right) \\
& =-\varepsilon^{M+1} U_{M-1}^{(4)}-\varepsilon^{M+2} U_M^{(4)},
\end{aligned}
$$
since $\{ U_j \}_{j \in \mathbb{N}_0}$ satisfies (\ref{eq:Uj}). By Theorem \ref{thm:main0} we get
$$
\begin{aligned}
& \left\|L\left(u-u_M^s\right)\right\|_{L^{\infty}(I)} \leq \\
\leq & C 4 ! K^4\left(\varepsilon^{M+1} \gamma^{M-1} \frac{a^{2(M-1)}(M-1)^{2(M-1)}}{(M-1) !}+\varepsilon^{M+2} \gamma^M \frac{a^{2 M} M^{2 M}}{M !}\right) \\
\leq & C K^4 \varepsilon^M \gamma^M \frac{a^{2 M} M^{2 M}}{M !}\left[\frac{1}{\gamma a^2}+\varepsilon\right] \\
\leq & C \varepsilon^M \gamma^M \frac{a^{2 M} M^M}{M !}.
\end{aligned}
$$
Using Stirling’s approximation we obtain
$$
\left\|L\left(u-u_M^s\right)\right\|_{L^{\infty}(I)} \leq C\left(a^2 \mathrm{e} \varepsilon \gamma M\right)^M.
$$

We next derive an estimate for $L \tilde{u}_M^{BL}$ :
$$
\begin{gathered}
L\left(\tilde{u}_M^{B L}\right)=\sum_{j=0}^{M+1} \varepsilon^j\left[\varepsilon^{-2}\left(\tilde{U}_j^{(4)}(\tilde{x})-\sum_{k=0}^j \varepsilon^k \alpha_k(\tilde{x}) \tilde{U}_{j-k}^{\prime \prime}(\tilde{x})-\sum_{k=0}^j \varepsilon^k \alpha_k^{\prime}(\tilde{x}) \tilde{U}_{j-k}^{\prime}(\tilde{x})\right)\right. \\
\left.+\sum_{k=0}^{j-2} \varepsilon^k \beta_k(\tilde{x}) \tilde{U}_{j-2-k}(\tilde{x})\right] \\
=\sum_{\substack{1 \leq k \leq j \leq M+1\\ j-2+k>M-1}}-\varepsilon^{j+k-2}\left(\alpha_k(\tilde{x}) \tilde{U}_{j-k}^{\prime \prime}(\tilde{x})+\alpha_k^{\prime}(\tilde{x}) \tilde{U}_{j-k}^{\prime}(\tilde{x})\right) \\
+\sum_{\substack{0 \leq k \leq j-2 \leq M+1\\ j+k>M-1}} \varepsilon^{k+j} \beta_k(\tilde{x}) \tilde{U}_{j-2-k}(\tilde{x}) .
\end{gathered}
$$
We utilize the analyticity of $\alpha$ and $\beta$  along with the estimate of Theorem \ref{thm:main0}, and with the aid
of Lemma \ref{lem:V} we obtain
$$
\begin{aligned}
\left\|L\left(\tilde{u}_M^{B L}\right)\right\|_{L^{\infty}(I)} \leq & \varepsilon^M \sum_{\substack{1 \leq k \leq j \leq M+1\\ j-2+k>M-1}}\left\{\left|\alpha_k(\tilde{x}) \tilde{U}_{j-k}^{\prime \prime}(\tilde{x})\right|+\left|\alpha_k^{\prime}(\tilde{x}) \tilde{U}_{j-k}^{\prime}(\tilde{x})\right|\right\} \\
& +\varepsilon^M \sum_{\substack{0 \leq k \leq j-2 \leq M+1\\ j+k>M-1}}\left|\beta_k(\tilde{x}) \tilde{U}_{j-2-k}(\tilde{x})\right|
\end{aligned}
$$
$$
\begin{aligned}
& \leq \varepsilon^M \sum_{\substack{1 \leq k \leq j \leq M+1\\ j-2+k>M-1}} C\left\{C_\alpha \gamma_\alpha^k|\tilde{x}|^k \gamma^{j-k} \frac{1}{(j-k-1) !}\left(a(j-k)+\frac{3}{\sqrt{\alpha(0)}}+|\tilde{x}|\right)^{2(j-k-1)}+\right. \\
& \left.\quad+C_{\alpha^{\prime}} \gamma_{\alpha^{\prime}}^k|\tilde{x}|^k \gamma^{j-k} \frac{1}{(j-k-1) !}\left(a(j-k)+\frac{2}{\sqrt{\alpha(0)}}+|\tilde{x}|\right)^{2(j-k-1)}\right\} \mathrm{e}^{-\sqrt{\alpha(0) \tilde{x}}}+ \\
& \quad+\varepsilon^M \sum_{\substack{0 \leq k \leq j-2 \leq M+1\\ j+k>M-1}} C_\beta \gamma_\beta^k|\tilde{x}|^k \gamma^{j-2-k} \frac{1}{(j-k-3) !}(a(j-2-k)+|\tilde{x}|)^{2(j-k-3)} \mathrm{e}^{-\sqrt{\alpha(0) \tilde{x}}} .
\end{aligned}
$$
Note that, for $j \ge 2$, there holds
$$
\begin{aligned}
& |\tilde{x}|^k \frac{1}{(j-k-1) !}\left(a(j+1-k)+\frac{3}{\sqrt{\alpha(0)}}+|\tilde{x}|\right)^{2(j-k-1)} \leq \frac{1}{(j-1) !}\left(a_j+\frac{3}{\sqrt{\alpha(0)}}+|\tilde{x}|\right)^{2(j-1)}, \\
& |\tilde{x}|^k \frac{1}{(j-k-3) !}\left(a(j-1-k)+\frac{3}{\sqrt{\alpha(0)}}+|\tilde{x}|\right)^{2(j-3-k)} \leq \frac{1}{(j-1) !}\left(a j+\frac{3}{\sqrt{\alpha(0)}}+|\tilde{x}|\right)^{2 j-3} .
\end{aligned}
$$
and hence
$$
\begin{aligned}
& |\tilde{x}|^k \frac{1}{(j-k-1) !}\left(a(j+1-k)+\frac{3}{\sqrt{\alpha(0)}}+|\tilde{x}|\right)^{2(j-k-1)} \leq \frac{1}{(j-1) !}\left(a_j+\frac{3}{\sqrt{\alpha(0)}}+|\tilde{x}|\right)^{2(j-1)}, \\
& |\tilde{x}|^k \frac{1}{(j-k-3) !}\left(a(j-1-k)+\frac{3}{\sqrt{\alpha(0)}}+|\tilde{x}|\right)^{2(j-3-k)} \leq \frac{1}{(j-1) !}\left(a j+\frac{3}{\sqrt{\alpha(0)}}+|\tilde{x}|\right)^{2 j-3} .
\end{aligned}
$$
Therefore,
$$
\begin{aligned}
& \frac{1}{(j-1) !}\left(a j+\frac{3}{\sqrt{\alpha(0)}}+|\tilde{x}|\right)^k \mathrm{e}^{-\sqrt{\alpha(0)} \tilde{x}} \leq C\left(a^2 \mathrm{e}\right)^{j-1} j^{j-1} \mathrm{e}^{\frac{6}{a \sqrt{\alpha(0)}}} \mathrm{e}^{-\frac{\sqrt{\alpha(0)}}{2} \tilde{x}}, \\
& \frac{1}{(j-3) !}\left(a j+\frac{3}{\sqrt{\alpha(0)}}+|\tilde{x}|\right)^k \mathrm{e}^{-\frac{\sqrt{\alpha(0)}}{2} \tilde{x}} \leq C\left(a^2 \mathrm{e}\right)^{j-3}(j-2)^{j-3} \mathrm{e}^{\frac{6}{a \sqrt{\alpha(0)}}} \mathrm{e}^{-\frac{\sqrt{\alpha(0)}}{2} \tilde{x}},
\end{aligned}
$$
and this gives
$$
\begin{aligned}
& \left\|L\left(\tilde{u}_M^{B L}\right)\right\|_{L^{\infty}(I)} \leq \varepsilon^M \sum_{\substack{1 \leq k \leq j \leq M+1\\ j-2+k>M-1}} C\left[C_\alpha \gamma_\alpha^k \gamma_*^{j-k}\left(a^2 \mathrm{e}\right)^{j-1} j^{j-1} \mathrm{e}^{-\frac{\sqrt{\alpha(0)}}{2}} \tilde{x}+\right. \\
& \left.+C_{\alpha^{\prime}} \gamma_{\alpha^{\prime}}^k \gamma_*^{j-k}\left(a^2 \mathrm{e}\right)^{j-1} j^{j-1} \mathrm{e}^{-\frac{\sqrt{\alpha(0)}}{2} \tilde{x}}\right]+ \\
& +\varepsilon^{M-1} \sum_{\substack{0 \leq k \leq j-2 \leq M+1\\ j+k>M-1}} C_\beta \gamma_\beta^k \gamma_*^{j-2-k}\left(a^2 \mathrm{e}\right)^{j-3}(j-2)^{j-3} \mathrm{e}^{-\frac{\sqrt{\alpha(0)}}{2} \tilde{x}} \\
& \leq \varepsilon^M C\left[\sum_{\substack{1 \leq k \leq j \leq M+1\\ j-2+k>M-1}} \gamma_*^j\left(a^2 \mathrm{e}\right)^{j-1} j^{j-1} \mathrm{e}^{-\frac{\sqrt{\alpha(0)}}{2}} \tilde{x}\left(C_\alpha\left(\frac{\gamma_\alpha}{\gamma_*}\right)^k+C_{\alpha^{\prime}}\left(\frac{\gamma_{\alpha^{\prime}}}{\gamma_*}\right)^k\right)+\right. \\
& \left.+\sum_{\substack{0 \leq k \leq j-2 \leq M+1\\j+k>M-1}} C_\beta \gamma_*^{j-2}\left(a^2 \mathrm{e}\right)^{j-3}(j-2)^{j-3} \mathrm{e}^{-\frac{\sqrt{\alpha(0)}}{2}} \tilde{x}\left(\frac{\gamma_\beta}{\gamma_*}\right)^k\right] . \\
&
\end{aligned}
$$
We recall that in Theorem \ref{thm:main0} the constant $\gamma_*$ is chosen in such a way so that
 the sums are bounded by convergent geometric series. Therefore,
$$
\begin{aligned}
\left\|L\left(\tilde{u}_M^{B L}\right)\right\|_{L^{\infty}(I)} & \leq C \varepsilon^M\left(a^2 \mathrm{e}\right)^M \gamma_*^{M+1}(M+1)^M \mathrm{e}^{-\sqrt{\alpha(0) \tilde{x}}} \\
& \leq C \gamma_*\left(a^2 \mathrm{e} \varepsilon \gamma_*(M+1)\right)^M .
\end{aligned}
$$
The term $\left\|L\left(\hat{u}_M^{B L}\right)\right\|_{L^{\infty}(I)}$ satisfies an analogous result. Thus we have
$$
\begin{aligned}
& \left\|L\left(r_M\right)\right\|_{L^{\infty}(I)} \leq \\
& \leq\left\|L\left(u-u_M^s\right)\right\|_{L^{\infty}(I)}+\left\|L\left(\tilde{u}_M^{B L}\right)\right\|_{L^{\infty}(I)}+\left\|L\left(\hat{u}_M^{B L}\right)\right\|_{L^{\infty}(I)} \\
& \leq C\left(a^2 \mathrm{e} \varepsilon \gamma_*(M+1)\right)^M .
\end{aligned}
$$
We choose $M + 1$ to be the integer part of $q/\varepsilon$, where $q =\frac{1}{\gamma_* a^2 e^2}$. Then we get
$$
a^2 \gamma_* \mathrm{e} \varepsilon(M+1) \leq \mathrm{e}^{-1}
$$
and $M \ge -2+q/\varepsilon$, therefore
$$
\left\|L\left(r_M\right)\right\|_{L^{\infty}(I I)} \leq C \mathrm{e}^{-M} \leq C \mathrm{e}^{-q / \varepsilon}
$$
We have shown that $r_M$ has exponentially small values at the endpoints of $[0, 1]$ and
$Lr_M$ is uniformly bounded by an exponentially small quantity on the interval $(0, 1)$.
By stability we have the desired result.
\end{proof}
%%%%%%%%%%%%%%%%%%%%%%%%%%%%%%%%%%%%%%%%%
\section{Appendix}

We present the proof of Proposition \ref{prop:lem_aux2}.

We determine the solution with the aid of a Green’s function. For $x \in (0, \infty)$, we seek for a solution 
of the form
$$
w(x)=\int_{-\infty}^{\infty} G(x, \xi) f(\xi) d \xi
$$
where $G$ is the Green's function. The solutions of the characteristic equation $m^4 - \lambda m^2 = 0$ are
$0, 0, \pm \kappa$. Hence, the general solution is
$$
w(x)=A+B x+C \mathrm{e}^{\kappa x}+D \mathrm{e}^{-\kappa x},
$$
for some  $A, B, C, D$, which depend on $\xi$. By enforcing continuity at $x = \xi$ for 
$G, \frac{\partial G}{\partial x}, \frac{\partial^2 G}{\partial x^2}$, and a jump condition for 
$\frac{\partial^3 G}{\partial x^3}$, we find 
$$
G(x, \xi)=\left\{\begin{array}{l}
-\frac{\xi}{\lambda}+\frac{x}{\lambda}-\frac{\mathrm{e}^{-\kappa \xi}}{2 \kappa \lambda} \mathrm{e}^{\kappa x}+\frac{2-\mathrm{e}^{-\kappa \xi}}{2 \kappa \lambda} \mathrm{e}^{-\kappa x}, \quad x<\xi \\
\frac{2-\mathrm{e}^{\kappa \xi}-\mathrm{e}^{-\kappa \xi}}{2 \kappa \lambda} \mathrm{e}^{-\kappa x}, \quad x>\xi
\end{array}\right. ,
$$
and the solution $w$ is given by
$$
\begin{aligned}
w(z)= & \frac{\mathrm{e}^{-\kappa z}}{\lambda^2} \int_0^{\infty} F(y / \kappa) d y-\frac{\mathrm{e}^{-\kappa z}}{2 \lambda^2} \int_0^{\infty} \mathrm{e}^{-y} F(y / \kappa) d y \\
& -\frac{\mathrm{e}^{-\sqrt{\lambda} z}}{2 \lambda^2} \int_0^{\kappa z} \mathrm{e}^y F(y / \kappa) d y-\frac{1}{\lambda^2} \int_{\kappa z}^{\infty} y F(y / \kappa) d y \\
& +\frac{z}{\kappa \lambda} \int_{\kappa z}^{\infty} F(y / \kappa) d y-\frac{\mathrm{e}^{\kappa z}}{2 \lambda^2} \int_{\kappa z}^{\infty} \mathrm{e}^{-y} F(y / \kappa) d y-\frac{g_1}{\kappa} \mathrm{e}^{-\kappa z} .
\end{aligned}
$$
We remove the restriction to $(0,\infty)$ with analytic continuation and we procced by
giving bounds on all the above terms. In order to bound the third term, we use as
path of integration the straight line connecting 0 and $\kappa z$ to get
$$
\begin{aligned}
\left|\frac{\mathrm{e}^{-\sqrt{\lambda} z}}{2 \lambda^2} \int_0^{\kappa z} \mathrm{e}^y F(y / \kappa) d y\right| & \leq \frac{1}{2\left|\lambda^2\right|} \mathrm{e}^{-\operatorname{Re}(\kappa z)} \int_0^1 C_F(q+t|z|)^{2 j-1}|\kappa z| \mathrm{e}^{-\operatorname{Re}(\kappa t z)} \mathrm{e}^{\operatorname{Re}(\kappa t z)} d t \\
& \leq C_F \frac{|\kappa|}{2\left|\lambda^2\right|} \frac{\mathrm{e}^{-\operatorname{Re}(\kappa z)}}{2 j}\left((q+|z|)^{2 j}-q^{2 j}\right) .
\end{aligned}
$$
The following, which gives an estimate for the sixth term, has an almost identical proof
to that of Lemma 7.3.6 in \cite{melenk}:
$$
\begin{aligned}
& \quad\left|\frac{\mathrm{e}^{\sqrt{\lambda} z}}{2 \lambda^2} \int_{\kappa z}^{\infty} \mathrm{e}^{-y} F(y / \kappa) d y\right|=\frac{1}{2|\lambda|^2}\left|\int_0^{\infty} \mathrm{e}^{-y} F(z+y / \kappa) d y\right| \\
& \quad \leq \frac{1}{2|\lambda|^2} \int_0^{\infty} \mathrm{e}^{-y} C_F \mathrm{e}^{-\operatorname{Re}(\kappa z+y)}\left(q+|z|+\frac{y}{|\kappa|}\right)^{2 j-1} d y \\
& \quad \leq \frac{1}{2|\lambda|^2} C_F \mathrm{e}^{-\operatorname{Re}(\kappa z)}|\kappa|^{-(2 j-1)} \int_0^{\infty} \mathrm{e}^{-2 y}(|\kappa| q+|\kappa z|+y)^{2 j-1} d y \\
& \leq \frac{1}{2|\lambda|^2} C_F \mathrm{e}^{-R e(\kappa z)}|\kappa|^{-(2 j-1)} \int_0^{\infty} \mathrm{e}^{-y} \frac{1}{2}(|\kappa| q+|\kappa z|+y / 2)^{2 j-1} d y \\
& \leq \frac{1}{2|\lambda|^2} C_F \mathrm{e}^{-R e(\kappa z)}|\kappa|^{-(2 j-1)} 2^{-2 j} \int_0^{\infty} \mathrm{e}^{-y}(2|\kappa| q+2|\kappa z|+y)^{2 j-1} d y \\
& =\frac{C_F}{2} \mathrm{e}^{-R e(\kappa z)}|\kappa|^{-(2 j-1)} 2^{-2 j} \mathrm{e}^{2|\kappa|(q+|z|)} \Gamma(2 j, 2|\kappa|(q+|z|)),
\end{aligned}
$$
where $\Gamma ( \cdot, \cdot)$ denotes the incomplete Gamma function. We observe that
$2 |\kappa| q \ge 2j - 1$. Thus we may employ the estimate
$$
\Gamma(a, \xi) \leq \frac{\mathrm{e}^{-\xi} \xi^a}{|\xi|-a_0}, \quad a_0=\max \{a-1,0\}, \operatorname{Re}(\xi) \geq 0,|\xi|>a_0,
$$
(which can be found in \cite[Chapter 4, Section 10]{Olver} to get
$$
\begin{aligned}
\left|\frac{\mathrm{e}^{\sqrt{\lambda} z}}{2|\lambda|^2} \int_{\kappa z}^{\infty} \mathrm{e}^{-y} F(y / \kappa) d y\right| & \leq \frac{1}{2|\lambda|^2} C_F \mathrm{e}^{-\operatorname{Re}(\kappa z)}|\kappa| \frac{(q+|z|)^{2 j}}{2|\kappa|(q+|z|)-(2 j-1)} \\
& \leq \frac{1}{2|\lambda|^2} C_F \mathrm{e}^{-\operatorname{Re}(\kappa z)}|\kappa| \frac{(q+|z|)^{2 j}}{6 j+1} .
\end{aligned}
$$
We note that the integral $\int_0^{\infty} \mathrm{e}^{-y} F(y / \kappa) d \xi$ is treated 
just like the sixth term. Hence, for the second term, we obtain
$$
\left|\frac{\mathrm{e}^{-\sqrt{\lambda} z}}{2 \lambda^2} \int_0^{\infty} \mathrm{e}^{-y} F(y / \kappa) d y\right| \leq \frac{1}{2|\lambda|^2} C_F \mathrm{e}^{-\operatorname{Re}(\kappa z)}|\kappa| \frac{(q+|z|)^{2 j}}{6 j+1}.
$$
The first and the last term of the solution satisfy the desired estimate as they are
multiplied by $e^{-\kappa z}$. To complete the proof we need to bound the term
$\frac{z}{\kappa \lambda} \int_{\kappa z}^{\infty} F(y / \kappa) d y-\frac{1}{\lambda^2} \int_{\kappa z}^{\infty} y F(y / \kappa) d y.$ We have
\begin{eqnarray*}
\frac{1}{|\lambda|^2}\left|\int_{\kappa z}^{\infty}(\kappa z-y) F(y / \kappa) d y\right|&=&
\frac{1}{|\lambda|^2}\left|\int_{\kappa z}^{\infty}(y-\kappa z) F(y / \kappa) d y\right|  
\end{eqnarray*}
\vspace{-0.25cm}
\begin{eqnarray*} 
&=& \frac{1}{|\lambda|^2}\left|\int_0^{\infty} t F\left(\frac{t}{\kappa}+z\right) d t\right| \\
&\leq& \frac{C_F \mathrm{e}^{-\operatorname{Re}(\kappa z)}}{|\lambda|^2} \int_0^{\infty} t\left(q+|z|+\frac{t}{|\kappa|}\right)^{2 j-1} \mathrm{e}^{-t} d t \\
&\leq& \frac{C_F \mathrm{e}^{-\operatorname{Re}(\kappa z)}}{|\lambda|^2} \sum_{n=0}^{2 j-1} \frac{1}{|\kappa|^n}\left(\begin{array}{c}
2 j-1 \\
n
\end{array}\right)(q+|z|)^{2 j-1-n} \int_0^{\infty} t^{n+1} \mathrm{e}^{-t} d t \\
&\leq&\frac{C_F \mathrm{e}^{-R e(\kappa z)}}{|\lambda|^2} \sum_{n=0}^{2 j-1} \frac{1}{|\kappa|^n}\left(\begin{array}{c}
2 j-1 \\
n
\end{array}\right)(q+|z|)^{2 j-1-n} \Gamma(n+2) \\
&\leq& \frac{C_F \mathrm{e}^{-R e(\kappa z)}}{|\lambda|^2}(q+|z|)^{2 j-1} \sum_{n=0}^{2 j-1} \frac{1}{|\kappa|^n} \frac{(n+1)(2 j-1) !}{(2 j-1-n) !}(q+|z|)^{-n} \\
&\leq& \frac{C_F \mathrm{e}^{-R e(\kappa z)}}{|\lambda|^2}(q+|z|)^{2 j-1} \sum_{n=0}^{2 j-1} \frac{(n+1)}{|\kappa|^n}(2 j-1)^n(q+|z|)^{-n} \\
&\leq& \frac{C_F \mathrm{e}^{-R e(\kappa z)}}{|\lambda|^2}(q+|z|)^{2 j-1} \sum_{n=0}^{2 j-1}(n+1)(q / 2)^n(q+|z|)^{-n} \\
&\leq& \frac{C_F \mathrm{e}^{-\operatorname{Re}(\kappa z)}}{|\lambda|^2}(q+|z|)^{2 j-1} \sum_{n=0}^{\infty} 2^{-n}(n+1) \\
&\leq& \frac{4 C_F \mathrm{e}^{-\operatorname{Re}(\kappa z)}}{|\lambda|^2}(q+|z|)^{2 j-1} \\
&\leq& \frac{C_F \mathrm{e}^{-\operatorname{Re}(\kappa z)}(q+|z|)^{2 j}}{j|\kappa \lambda|} .
\end{eqnarray*}

Combining all the above inequalities we get
$$
|w(z)| \leq C\left[\frac{C_F}{j}(q+|z|)^{2 j}+\frac{\left|g_1\right|}{|\kappa|}\right] \mathrm{e}^{-R e(\kappa z)}, \quad \text { for all } z \in \mathbb{C}.
$$

%%%% Bibliography  %%%%%%%%%%

\end{document}